\documentclass[oneside,english]{amsart}
\usepackage[T1]{fontenc}
\usepackage[latin9]{inputenc}
\usepackage{amsthm}
\usepackage{amssymb}
\usepackage[authoryear]{natbib}

\makeatletter
\numberwithin{equation}{section}
\numberwithin{figure}{section}
\theoremstyle{plain}
\newtheorem{thm}{\protect\theoremname}[section]
\theoremstyle{plain}
\newtheorem{prop}[thm]{\protect\propositionname}
\theoremstyle{plain}
\newtheorem{cor}[thm]{\protect\corollaryname}
\theoremstyle{plain}
\newtheorem{lem}[thm]{\protect\lemmaname}

\makeatother

\usepackage{babel}
\providecommand{\corollaryname}{Corollary}
\providecommand{\lemmaname}{Lemma}
\providecommand{\propositionname}{Proposition}
\providecommand{\theoremname}{Theorem}

\begin{document}
\title{Coefficient growth in square chains}
\author{Shawn Walker}
\begin{abstract}
Suppose $((\cdots((x^{2}-c_{1})^{2}-c_{2})^{2}\cdots)^{2}-c_{k-1})^{2}-c_{k}$
splits into linear factors over $\mathbb{Z}$ and $c_{k}\neq0$. We
show that for each $j$ and each prime $p$, if $p\leq2^{j-1}$ then
$p$ divides $c_{j}$. Consequently,
\[
\ln c_{j}>\frac{1}{4}\cdot2^{j}\,\,\mathrm{for}\,j\geq5
\]
If we also have $p\equiv3\,(\mathrm{mod\,4)}$ then $p^{2^{j-\left\lceil \lg p\right\rceil }}$
divides $c_{j}$. Consequently, if $k\geq3$, there exists some absolute
constant $\lambda>0$ so that,
\[
\ln c_{j}>\lambda k2^{j}\mathrm{\,\,for\,all\,}j
\]
These estimates argue against the possibility of explicitly constructing
polynomials of the given form for large $k$, as the coefficients
quickly become too large to manipulate.
\end{abstract}

\maketitle

\section{\label{sec:factoring}Motivation: factoring integers with square
chains}

Call a polynomial of the form 
\[
P(x)=((\cdots((x^{2}-c_{1})^{2}-c_{2})^{2}\cdots)^{2}-c_{k-1})^{2}-c_{k}
\]
a \textbf{square chain} of length $k$. Some square chains of lengths
$k=3,4$ are presented in \citet[research problem 6.18]{crandall2005}
which have the property that they have $2^{k}$ distinct integer roots.
Crandall and Pomerance then ask about the existence of longer square
chains, suggesting that sufficiently long chains might be useful for
factoring large integers. Indeed, a simple scheme shows promise:

Suppose $n=pq$ is an odd semiprime, and that $p$ and $q$ are approximately
the same size. If $P(x)\in\mathbb{Z}[x]$ has about $\frac{p}{2}\approx\frac{q}{2}$
distinct roots, one can reasonably hope that $P(m)\equiv0\,(\mathrm{mod\ }p)$
for about half of $m\in\{0,1,\ldots,n-1\}$, and similarly hope $P(m)\equiv0\,(\mathrm{mod\ }q)$
for about half of $m\in\{0,1,\ldots,n-1\}$. If we assume heuristically
that these are independent events, then about a quarter of the choices
of $m$ yield $\gcd(P(m),n)=p$, and about a quarter yield $\gcd(P(m),n)=q$.
For a more rigorous analysis, see \citet{DBLP:conf/ants/Lipton94}.

To make this a tractable factoring algorithm, we need a polynomial
$P$ which can be efficiently evaluated and has sufficiently many
distinct roots. Square chains are nearly ideal from the standpoint
of efficient evaluation. As a polynomial of degree $2^{k}$, a square
chain of length $k$ may have up to $2^{k}$ roots and may be evaluated
using only $k$ multiplications. No polynomial of degree $2^{k}$
can be evaluated with fewer multiplications (\citet{DBLP:journals/cjtcs/BorchertMR13}).
An obstacle is finding square chains with many distinct roots.

Let us set aside the question of existence. Suppose that there exists
a square chain of length $k$ which has exactly $2^{k}$ roots, counting
multiplicity, even if those roots are not all distinct. What properties
might such a square chain have?

We say a polynomial $P$ \textbf{crumbles} over the unique factorization
domain $D$ if $P$ may be written as the product of (not necessarily
distinct) linear polynomials in $D[x]$. Unless otherwise indicated
by context, we will assume $D=\mathbb{Z}.$ Clearly $x^{2^{k}}=(\cdots(x^{2}-0)^{2}\cdots-0)^{2}-0$
crumbles over any UFD. Unfortunately, it has only 1 distinct root.
More generally, if we have any crumbling square chain $P$, then $P(x)^{2}-0$
is a crumbling square chain longer than $P$. But extending a chain
in this manner doesn't create any new roots. Conversely, any square
chain whose final coefficient is 0 has exactly the same root set as
a square chain of shorter length, and we can instead consider the
shorter square chain. To this end, call a square chain whose final
coefficient is nonzero a \textbf{fundamental} square chain.

While it may be a big ask, suppose we are able to find some fundamental
crumbling square chain $P$. Then it is guaranteed to have plenty
of distinct roots.
\begin{prop}
\label{prop:ufd}Let $D$ be a unique factorization domain with $1_{D}+1_{D}\neq0_{D}$.
Suppose $P(x)=(\cdots(x^{2}-c_{1})^{2}-\cdots)^{2}-c_{k}\in D[x]$
crumbles over $D$. If $c_{k}\neq0_{D}$, then $P$ has at least $2^{k-1}+1$
distinct roots.
\end{prop}

\begin{proof}
Let $T_{k}=\{c_{k}\}$, and call this the $k$-th tail square set
of $P$. Note that $c_{k}$ must be a perfect square; indeed, for
any root $r$ of $P$, we have $c_{k}=(\cdots(r^{2}-c_{1})^{2}\cdots-c_{k-1})^{2}$.
So we may use the identity $a^{2}-b^{2}=(a+b)(a-b)$, to split $P(x)$
into two square chains: 
\[
P(x)=((\cdots(x^{2}-c_{1})^{2}-\cdots)^{2}-c_{k-1}-\sqrt{c_{k}})((\cdots(x^{2}-c_{1})^{2}-\cdots)^{2}-c_{k-1}+\sqrt{c_{k}})
\]
Since $P$ crumbles and $D[x]$ is a unique factorization domain,
each of these factors must crumble. Let $T_{k-1}=\{c_{k-1}\pm t:t^{2}\in T_{k}\}$
be the tail squares of the two factors. Exactly as was the case with
$T_{k}$, elements of $T_{k-1}$ must be perfect squares.

We then repeat this splitting with each factor, splitting $P$ into
4 factors with tail square set $T_{k-2}=\{c_{k-2}\pm t:t^{2}\in T_{k-1}\}$,
then 8 factors with tail square set $T_{k-3}=\{c_{k-3}\pm t:t^{2}\in T_{k-2}\}$,
and so forth onto $2^{k-1}$ factors of the form $(x^{2}-a)$ with
$a\in T_{1}=\{c_{1}\pm t:t^{2}\in T_{2}\}$. If we make the convention
that $c_{0}=0$, we may split the previous set of factors one more
time to get $2^{k}$ factors of the form $(x-r)$ with $r$ in the
set $T_{0}=\{0\pm t:t^{2}\in T_{1}\}$. That is, $T_{0}$ is the set
of roots of $P$.

Suppose $j\in\{0,1,\ldots,k-1\}$. What can we say about the size
of $T_{j}$? Distinct elements of $T_{j+1}$ yield distinct elements
of $T_{j}$: if $\{c_{j}\pm t\}\cap\{c_{j}\pm s\}\neq\emptyset$ then
$t=\pm s$, and so $t^{2}=s^{2}$. Each nonzero element of $T_{j+1}$
yields two distinct elements of $T_{j}$: if $c_{j}+t=c_{j}-t$, then
$2t=0$, and so $t=0$ since $D$ is an integral domain and $2\neq0$.
As such,
\[
|T_{j}|=\left\{ \begin{array}{rl}
2(|T_{j+1}|-1)+1 & \mathrm{if\,}0\in T_{j+1}\\
2|T_{j+1}| & \mathrm{if\,}0\not\in T_{j+1}
\end{array}\right\} \geq2(|T_{j+1}|-1)+1
\]
Consequently, 
\[
|T_{0}|\geq2(|T_{1}|-1)+1\geq2^{2}(|T_{2}|-1)+1\geq\cdots\geq2^{k-1}(|T_{k-1}|-1)+1
\]

To complete the argument, we note that $|T_{k-1}|=2|T_{k}|=2$ since
$0\not\in T_{k}$ by hypothesis.
\end{proof}
As an aside, the bound of proposition~\ref{prop:ufd} is met exactly
when $D=\mathbb{Z}_{q}$ for a prime $q$ satisfying $q\equiv1\,(\mathrm{mod\ }2^{k-1})$,
and when the square chain $P$ of length $k$ is given by
\begin{align*}
P(x) & =(\cdots((x^{2}-0)^{2}\cdots-0)^{2}-2^{-1})^{2}-2^{-2}\\
 & =(x^{2^{k-1}}-2^{-1})^{2}-2^{-2}\\
 & =(x^{2^{k-1}}-1)(x^{2^{k-1}}-0)
\end{align*}
In this case, $0$ is a root of $P$ with multiplicity $2^{k-1}$,
and each $2^{k-1}$-th root of unity $(\mathrm{mod\ }q)$ is a root
of $P$ with multiplicity 1.

\section{Coefficient growth}

Perhaps more interesting than proposition~\ref{prop:ufd} is a specialization
of its contrapositive: for an odd characteristic finite field, fundamental
crumbling square chains must be of strictly bounded length.
\begin{cor}
\label{cor:primechain}Suppose $P(x)=(\cdots(x^{2}-c_{1})^{2}-\cdots)^{2}-c_{k}\in F[x]$
crumbles over the finite field $F$ with $\mathrm{char}(F)\neq2$.
If $|F|\leq2^{j-1}$ then 
\[
c_{j}=c_{j+1}=\cdots=c_{k}=0_{F}
\]
\end{cor}

\begin{proof}
Suppose $c_{i}$ is the last non-zero coefficient in $P$, so that
$Q(x)=(\cdots(x^{2}-c_{1})^{2}-\cdots)^{2}-c_{i}$ is a fundamental
crumbling square chain over $F$. By proposition \ref{prop:ufd},
$Q$ must have at least $2^{i-1}+1$ distinct roots in $F$. Of course,
$2^{i-1}+1\leq|F|$, as $Q$ can have, at most, all of $F$ as roots.
So $2^{i-1}+1\leq|F|\leq2^{j-1}.$ Consequently, $i<j$, and so all
coefficients of index $j$ or larger must be zero.
\end{proof}
Our primary interest in corollary~\ref{cor:primechain} will be in
the case $F=\mathbb{Z}_{p}$ for $p$ any prime. But corollary~\ref{cor:primechain}
does not cover the case of $\mathbb{Z}_{2}$. Instead, we derive a
similar, if weaker, result to cover $\mathbb{Z}_{2}$.
\begin{lem}
\label{lem:z2}Suppose $P(x)=(\cdots(x^{2}-c_{1})^{2}-\cdots)^{2}-c_{k}\in\mathbb{Z}[x]$
crumbles over $\mathbb{Z}$. Then $c_{2}\equiv c_{3}\equiv\cdots\equiv c_{k}\equiv0\,(\mathrm{mod\ }2)$
\end{lem}

\begin{proof}
Consider the tail square sets $T_{j}$ from proposition~\ref{prop:ufd}.
Let us adopt the convention that $T_{k+1}=\{0\}$. This is consistent
with our previous definition, as it makes $T_{k}=\{c_{k}\pm t:t^{2}\in T_{k+1}\}=\{c_{k}\}$.

For an arbitrary $j>1$, choose an arbitrary $t^{2}\in T_{j}$. Then
$c_{j-1}\pm t\in T_{j-1}$ by definition, and so $c_{j-1}\pm t$ are
both squares. Thus,
\[
c_{j-1}\pm t\equiv0,1,\mathrm{\,or}\,4\,(\mathrm{mod\ }8)
\]
The limited set of congruence classes that these elements fall into
lets us somewhat limit what congruence classes their difference falls
into. That is,
\[
(c_{j-1}+t)-(c_{j-1}-t)\equiv2t\equiv0,1,3,4,5,\mathrm{\,or\,}7\,(\mathrm{mod\ }8),
\]
This limits the congruence classes $t$ may fall into as well:
\[
t\equiv0,2,4,\mathrm{\,or\,}6\,(\mathrm{mod\ }8).
\]
So $t$ must be even. Since $t$ was chosen arbitrarily, each element
$t^{2}\in T_{j}$ must be the square of an even number.

For a given $j$ with $1<j<k+1$, choose any $r^{2}\in T_{j+1}$.
By the previous argument, $r$ is even. We have $(c_{j}+r)\in T_{j}$
and so $(c_{j}+r)$ is even. Consequently, $c_{j}$ must be even.
\end{proof}
As a consequence of corollary~\ref{cor:primechain}, coefficients
in long crumbling square chains over $\mathbb{Z}$ must have many
prime factors. This implies a lower bound on the size of those coefficients.
To quantify this, define the primorial of $m$ as:
\[
m\#=\prod_{\begin{array}{c}
p\mathrm{\,prime}\\
p\leq m
\end{array}}p
\]

\begin{prop}
\label{prop:growth}Suppose $P(x)=(\cdots(x^{2}-c_{1})^{2}-\cdots)^{2}-c_{k}\in\mathbb{Z}[x]$
crumbles over $\mathbb{Z}$. Then for each $j$, we have $2^{j-1}\#$
divides $c_{j}$. If $c_{k}\neq0$ and $j\geq5$, then $\ln c_{j}>2^{j-2}$.
\end{prop}

\begin{proof}
$c_{j}$ is even for $j\geq2$ by lemma~\ref{lem:z2}. For each odd
prime $p\leq2^{j-1}$, corollary~\ref{cor:primechain} tells us that
$p$ divides $c_{j}$. Thus $2^{j-1}\#$ divides $c_{j}$.

Suppose $c_{k}>0.$ We claim that $c_{j}>0$ for each $j$. Clearly
$c_{j}\geq0.$ To see that $c_{j}\neq0$, let $i$ be an arbitrary
index $1\leq i<k$. We may write the equation $P(x)=0$ as

\[
(\cdots(x^{2}-c_{1})^{2}-\cdots)^{2}=c_{i}\pm\sqrt{\cdots\pm\sqrt{c_{k}}}
\]

Since the right-hand side of this equation must be non-negative for
every choice of signs, we have $c_{i}\geq\sqrt{c_{i+1}+\cdots}\geq\sqrt{c_{i+1}}$.
As the choice of $i$ was arbitrary, we may apply this inequality
recursively, yielding $c_{j}\geq\sqrt[2^{k-j}]{c_{k}}$. But $c_{k}\neq0$;
it follows that $c_{j}\neq0$ as well.

Together with the fact that $2^{j-1}\#$ divides $c_{j}$, the positivity
of $c_{j}$ implies that $2^{j-1}\#\leq c_{j}$.

\citet[equation 3.14]{rosser1962} establish:
\[
\begin{array}{cc}
x(1-\frac{1}{2\ln x})\leq\ln(x\#) & \mathrm{for\,}x\geq563\end{array}
\]

An exhaustive calculation (omitted) demonstrates the looser bound
\[
\begin{array}{lr}
\frac{1}{2}x<\ln(x\#) & \mathrm{for\,}11\leq x\leq563\end{array}
\]

Thus 
\[
\begin{array}{lr}
\frac{1}{2}2^{j-1}<\ln(2^{j-1}\#)\leq\ln c_{j} & \mathrm{if\,}11\leq2^{j-1}\end{array}
\]
\end{proof}

\section{An asymptotic refinement}

It seems unlikely the lower bounds given in proposition~\ref{prop:growth}
are the best possible. $\ln(x\#)\sim x$, and so it seems likely,
at the very least, that $\ln c_{j}\geq2^{j-1}$ for $j$ sufficiently
large. At the same time, it is plausible that for some primes $p$,
not only must $p$ divide $c_{j}$, but possibly $p^{i}$ divides
$c_{j}$ for some appropriate condition on $p,$ $i$ and $j$. Indeed,
some reflection shows that primes of the form $4n+3$ must be much
more prevalent than we've heretofore indicated.

Let $\nu_{p}(n)=\max\{e\in\mathbb{Z}:p^{e}\big|n\}$ be the exponent
of $p$ in the prime factorization of $n$.

As noted by \citet{DBLP:journals/em/Dilcher00}, at least some of
the coefficients in a crumbling square chain must be expressable as
half the sum of two squares. We show that every coefficient in a crumbling
square chain must be so expressable. By the sum of two squares theorem,
if $p\equiv3\,(\mathrm{mod\,}4)$, then $\nu_{p}(a^{2}+b^{2})$ must
be even. And so $\nu_{p}(c_{j})$ must be even for every coefficient
$c_{j}$ in a crumbling square chain.

What's more, by similar considerations as go into the sum of two squares
theorem, we can propagate powers forward to following coefficients,
so that $\nu_{p}(c_{j+1})\geq2\nu_{p}(c_{j})$.

We collect these ideas into the following lemma.
\begin{lem}
\label{lem:accum3m4}Suppose $P(x)=(\cdots(x^{2}-c_{1})^{2}-\cdots)^{2}-c_{k}\in\mathbb{Z}[x]$
crumbles over $\mathbb{Z}$. If $p$ is prime with $p\equiv3\,(\mathrm{mod\ }4)$,
and $p\leq2^{j-1}$ then $\nu_{p}(c_{j})\geq2^{j-\left\lceil \lg p\right\rceil }$
.
\end{lem}

\begin{proof}
Suppose $\left\lceil \lg p\right\rceil =h-1$. Then $p$ divides $c_{h}$
by corollary~\ref{cor:primechain}.

Choose an arbitrary $t^{2}\in T_{h+1}$. By definition, $c_{h}\pm t\in T_{h}$.
All members of $T_{h}$ are squares, so there exist $r,s$ so that
$c_{h}+t=r^{2}$ and $c_{h}-t=s^{2}$. Then $2c_{h}=r^{2}+s^{2}$,
that is $2c_{h}$ is the sum of two squares.

By the sum of two squares theorem, since $p\equiv3\,(\mathrm{mod\,}4)$
and $p$ divides $r^{2}+s^{2}$, then $p^{2}$ divides both $r^{2}$
and $s^{2}$. It follows that $\nu_{p}(c_{h})\geq2^{1}=2^{h-\left\lceil \lg p\right\rceil }$.

To handle the more general case, we proceed inductively. Suppose $\nu_{p}(c_{j-1})\geq2^{j-1-\left\lceil \lg p\right\rceil }$.
We will show that $\nu_{p}(c_{j})\geq2^{j-\left\lceil \lg p\right\rceil }$.

As before, choose an arbitrary $t^{2}\in T_{j}$, making $c_{j-1}\pm t\in T_{j-1}$.
Write $c_{j-1}+t=r^{2}$ and $c_{j-1}-t=s^{2}$. Then $2c_{j-1}=r^{2}+s^{2}$.
Since $p\neq2$,
\begin{align*}
\nu_{p}(r^{2}+s^{2}) & =\nu_{p}(2c_{j-1})\\
 & =\nu_{p}(c_{j-1})\\
 & \geq2^{j-1-\left\lceil \lg p\right\rceil }
\end{align*}

Also, we may write $r^{2}+s^{2}=(r+is)(r-is)$ as the product of two
Gaussian integers. We recall two well known results: the Gaussian
integers form a unique factorization domain, and $p$ is a prime Gaussian
integer since $p\equiv3\,(\mathrm{mod\ }4)$. As such, it makes sense
to extend our definition of $\nu_{p}$ to Gaussian integers. Now,
$\nu_{p}(r+is)=\nu_{p}(r-is)$ since $p$ is its own complex conjugate.
Thus 
\begin{align*}
2\nu_{p}(r+is) & =\nu_{p}(r+is)+\nu_{p}(r-is)\\
 & =\nu_{p}((r+is)(r-is))\\
 & =\nu_{p}(r^{2}+s^{2})\\
 & \geq2^{j-1-\left\lceil \lg p\right\rceil }
\end{align*}
So 
\[
\nu_{p}(r+is)\geq\frac{1}{2}2^{j-1-\left\lceil \lg p\right\rceil }
\]
Any power of $p$ that divides $r+is$ must divide both $r$ and $s$,
so
\[
\nu_{p}(r),\,\nu_{p}(s)\geq\frac{1}{2}2^{j-1-\left\lceil \lg p\right\rceil }
\]
implying that
\[
\nu_{p}(r^{2}),\,\nu_{p}(s^{2})\geq2^{j-1-\left\lceil \lg p\right\rceil }
\]

Since $r^{2}$ and $s^{2}$ share a common power of $p$, their difference
$r^{2}-s^{2}=2t$ must share the same common power. That is, 
\begin{align*}
\nu_{p}(2t) & =\nu_{p}(r^{2}-s^{2})\\
 & \geq2^{j-1-\left\lceil \lg p\right\rceil }
\end{align*}
Since $p\neq2$, this implies that 
\[
\nu_{p}(t)\geq2^{j-1-\left\lceil \lg p\right\rceil }
\]
 and so 
\[
\nu_{p}(t^{2})\geq2^{j-\left\lceil \lg p\right\rceil }
\]
As $t^{2}$ was chosen arbitrarily, this inequality holds for any
$t^{2}\in T_{j}$.

Choose any $u^{2}\in T_{j+1}$. We have $c_{j}\pm u\in T_{j}$ by
definition. So 
\[
\nu_{p}(c_{j}+u),\,\nu_{p}(c_{j}-u)\geq2^{j-\left\lceil \lg p\right\rceil }
\]
Since $c_{j}+u$ and $c_{j}-u$ share a common power of $p$, their
sum, $(c_{j}+u)+(c_{j}-u)=2c_{j}$ must share the same common power.
That is,
\begin{align*}
\nu_{p}(2c_{j}) & =\nu_{p}((c_{j}+u)+(c_{j}-u)).\\
 & \geq2^{j-\left\lceil \lg p\right\rceil }
\end{align*}
And since $p\neq2$, 
\[
\nu_{p}(c_{j})\geq2^{j-\left\lceil \lg p\right\rceil }
\]
\end{proof}
By proposition~\ref{prop:growth}, later coefficients ``pick up''
many primes as divisors. By lemma~\ref{lem:accum3m4}, once a coefficient
acquires a divisor $p\equiv3\,(\mathrm{mod\ }4)$, each later coefficient
is divisible many times by the same prime. Define 
\[
x\#_{3:4}=\prod_{\begin{array}{c}
p\,\mathrm{prime}\\
p\equiv3\,(\mathrm{mod\ }4)\\
p\leq x
\end{array}}p
\]

Together, proposition~\ref{prop:growth} and lemma~\ref{lem:accum3m4}
imply:
\begin{prop}
\label{prop:growth3m4}Suppose $P(x)=(\cdots(x^{2}-c_{1})^{2}-\cdots)^{2}-c_{k}\in\mathbb{Z}[x]$
crumbles over $\mathbb{Z}$. Let 
\[
D_{j}=2^{j-1}\#\cdot\prod_{i=0}^{j-1}(2^{j-1-i}\#_{3:4})^{2^{i}}
\]
 For each $j$, $D_{j}$ divides $c_{j}$. If $c_{k}\neq0$, then
for some absolute constant $\lambda>0$ and each $j\geq3$,
\[
\ln c_{j}>\lambda j2^{j}
\]
\end{prop}

\begin{proof}
For primes $p\leq2^{j-1}$:

If $p\not\equiv3\,(\mathrm{mod\ }4)$ then $\nu_{p}(D_{j})=1$.

If $p\equiv3\,(\mathrm{mod\ }4)$ then 
\[
\nu_{p}(D_{j})=1+\sum_{i=0}^{j-1-\left\lceil \lg p\right\rceil }2^{i}=2^{j-\left\lceil \lg p\right\rceil }
\]

Thus by corollary~\ref{cor:primechain} and lemma~\ref{lem:accum3m4},
$D_{j}$ divides $c_{j}$.

Corollaries of the Siegel-Walfisz theorem give $\ln(x\#_{3:4})\sim\frac{1}{2}x$,
c.f. \citet[corollaries 11.15, 11.20]{montgomery_vaughan_2006}. So
there must be some constant $\lambda_{3:4}>0$ satisfying $\lambda_{3:4}x<\ln(x\#_{3:4})$
for all $x\geq3$. Similarly, by the prime number theorem, $\ln(x\#)\sim x$,
and thus there is a constant $\lambda_{1}$ so that $\lambda_{1}x<\ln(x\#)$
for all $x\geq2$. Let $4\lambda=\min(\lambda_{1},\lambda_{3:4})$.
Then for $j\geq3$

\[
\begin{array}{rcl}
\ln D_{j} & = & \ln(2^{j-1}\#)+{\displaystyle {\sum_{i=0}^{j-1}}}2^{i}\cdot\ln(2^{j-1-i}\#_{3:4})\\
 & = & \ln(2^{j-1}\#)+{\displaystyle {\sum_{i=0}^{j-3}}2^{i}\cdot\ln(2^{j-1-i}\#_{3:4})}\\
 & > & 4\lambda2^{j-1}+{\displaystyle 4\lambda{\sum_{i=0}^{j-3}}2^{i}\cdot2^{j-1-i}}\\
 & = & 4\lambda2^{j-1}(1+j-2)\\
 & = & 2\lambda(j-2)2^{j}\\
 & \geq & \lambda j2^{j}
\end{array}
\]

If $c_{k}\neq0$, then $c_{j}>0$ as argued in proposition~\ref{prop:growth}.
Since $D_{j}$ divides $c_{j}$, positivity of $c_{j}$ forces $D_{j}\leq c_{j}$.
And so $\lambda j2^{j}<\ln D_{j}\leq\ln c_{j}$.
\end{proof}
We can sharpen the closed form estimate slightly:
\begin{cor}
Suppose $P(x)=(\cdots(x^{2}-c_{1})^{2}-\cdots)^{2}-c_{k}\in\mathbb{Z}[x]$
crumbles over $\mathbb{Z}$, $k\geq3$, and $c_{k}>0$. Then there
exists an absolute constant $\lambda>0$ so that for all $j$,
\[
\ln c_{j}>\lambda k2^{j}
\]
\end{cor}

\begin{proof}
We observed in the proof of proposition~\ref{prop:growth} that $c_{j}\geq\sqrt[2^{k-j}]{c_{k}}$
so,
\begin{align*}
\ln c_{j} & \geq\frac{\ln c_{k}}{2^{k-j}}\\
 & >\frac{\lambda k2^{k}}{2^{k-j}}\\
 & =\lambda k2^{j}
\end{align*}
\end{proof}
It is interesting to note that this statement is enough to show the
length of the chain influences the size of $c_{1}$, as it shows that
$\ln c_{1}\geq2\lambda k$.

\section{Discussion and related work}

To factor a product of two 500-bit primes using the algorithm described
in section~\ref{sec:factoring}, we would need to start by constructing
a fundamental crumbling square chain of length not much smaller than
500. According to proposition~\ref{prop:growth}, the coefficient
$c_{400}$ of such a chain would be at least $\lg e^{2^{398}}\approx9.3\times{10}^{119}$
bits in length. By way of comparison, estimates place the total digital
storage capacity of the world at approximately ${10}^{22}$ bits as
of the year 2019. Even if we knew how to construct such a chain, precalculating
the coefficients of such a chain would clearly be infeasible.

Finding fundamental crumbling square chains has proven difficult.
While \citet{DBLP:journals/em/Dilcher00} provides a characterization
of length 3 fundamental crumbling square chains, and \citet{Bremner08}
describes two infinite families of length 4 fundamental crumbling
square chains, no fundamental crumbling square chains of length 5
are known. Indeed, \citet{DBLP:journals/cjtcs/BorchertMR13} points
out that a crumbling square chain of length 5 with distinct roots
would advance understanding of a historied question known as the Prouhet-Tarry-Escott
problem.

\citet{DBLP:journals/cjtcs/BorchertMR13} discuss a more general family
of polynomials, which they term gems. By construction, their gems
are polynomials which are efficiently computable, crumble over $\mathbb{Z}$,
and have distinct roots. While the highest known degree of a square
chain that crumbles over $\mathbb{Z}$ is 16, the authors of that
article describe gems of degrees up to 55.

\bibliographystyle{plainnat}
\bibliography{squarechain}

\end{document}